\documentclass[12pt, a4paper, english, reqno]{amsart}

\usepackage{a4wide}

\usepackage[utf8]{inputenc}
\usepackage{amsmath}
\usepackage{amsthm,amssymb}
\usepackage{nicefrac,amsfonts}
\usepackage{thmtools,mathtools,leftindex,leftidx}
\usepackage{caption,subcaption}
\usepackage{epsfig,graphicx,graphics,color,tikz,tikz-cd,tcolorbox}
\usepackage{enumerate,enumitem}
\usepackage[sort,nocompress]{cite}
\usepackage{xspace}
\usepackage{bbm}
\usepackage{hyperref}
\usepackage[capitalise]{cleveref}
\usepackage{float}
\hypersetup{
    colorlinks=true,       
    linkcolor=blue,        
    citecolor=magenta,         
    filecolor=magenta,     
    urlcolor=cyan,         
    linktoc=all
}

\makeatletter
\renewcommand{\paragraph}{%
  \@startsection{paragraph}{4}%
  {\z@}{1.6ex \@plus 1ex \@minus .2ex}{-0.5em}%
  {\normalfont\normalsize\bfseries}%
}
\makeatother

\theoremstyle{plain}

\newtheorem{theorem}{Theorem}[section]
\newtheorem{proposition}[theorem]{Proposition}
\newtheorem{corollary}[theorem]{Corollary}

\theoremstyle{definition}

\newtheorem{remark}[theorem]{Remark}

\newtheorem{definition}[theorem]{Definition}
\newtheorem{example}[theorem]{Example}

\newcommand*{\claimproofname}{Proof of claim.}

\def\final{1}  
\def\iflong{\iffalse}
\ifnum\final=0  
\newcommand{\adam}[1]{{\color{red}[{ \textbf{Adam:}  #1}]\marginpar{\color{red}*}}}
\newcommand{\lorenzo}[1]{{\color{teal}[{ \textbf{Lorenzo:}  #1}]\marginpar{\color{teal}*}}}
\else 
\newcommand{\adam}[1]{}
\newcommand{\lorenzo}[1]{}
\fi

\DeclareMathOperator\len{len}

\newcommand{\bN}{\mathbb{N}}

\newcommand{\bQ}{\mathbb{Q}}
\newcommand{\bR}{\mathbb{R}}

\newcommand{\cC}{\mathcal{C}}

\newcommand{\zero}{\hat{0}}
\newcommand{\one}{\hat{1}}
\newcommand{\FY}{\operatorname{FY}}
\newcommand{\SFY}{\operatorname{SFY}}
\newcommand{\SI}{\operatorname{SI}}

\let\Right\bigr
\let\Left\bigl
\makeatletter
\def\bigr#1{\Right#1\@ifnextchar){\!\bigr}{}}
\def\bigl#1{\Left#1\@ifnextchar({\!\bigl}{}}
\makeatother

\title[Kähler algebras and Chow polynomials]{Existence of Kähler algebras with \\ Chow polynomials as Hilbert series}
\author[A. Schweitzer]{Adam Schweitzer}
\address{Department of Mathematics, KTH Royal Institute of Technology, Stockholm, Sweden}
\email{adasch@kth.se}
\author[L. Vecchi]{Lorenzo Vecchi}
\address{Department of Mathematics, KTH Royal Institute of Technology, Stockholm, Sweden}
\email{lvecchi@kth.se}


\begin{document}

\begin{abstract}
In this article, we study Chow polynomials of weakly ranked posets and prove the existence of Gorenstein algebras with the Kähler package such that their Hilbert--Poincaré series agrees with the Chow polynomial.
Our statement provides evidence in support of a conjecture by Ferroni, Matherne and the second author about the existence of an algebra for every weakly ranked poset that generalizes the Feichtner--Yuzvinsky Chow ring for matroids.
This allows us to prove strong inequalities for the coefficients of Chow polynomials; we prove log-concavity for all posets of weak rank at most six and provide counterexamples to log-concavity for any higher rank.
For ranked posets we recover an even stronger condition, showing that the differences between consecutive coefficients constitute a pure $O$-sequence.
\end{abstract}

\maketitle

\section{Introduction}
A central theme in the study of enumerative invariants of posets is the extent to which integer sequences satisfy strong regularity properties, such as unimodality, palindromicity, log-concavity, or the existence of combinatorial or algebraic models realizing them as Hilbert--Poincaré series. A classical example of such a construction is due to Stanley \cite{gThmNecessity} and independently McMullen \cite{gThmNecessity2}, who construct for every simplicial polytope $P$ a graded ring whose Hilbert--Poincaré series coincides with the \emph{$h$-vector} of $P$.
Chow polynomials of partially ordered sets, introduced recently by Ferroni, Matherne, and the second author, provide a unifying framework for encoding subtle enumerative and algebraic properties of posets \cite{ferroni-matherne-vecchi}. The name comes from the theory of combinatorial algebraic geometry for matroids; when the poset is a geometric lattice, it coincides with the Hilbert--Poincaré series of the Chow ring of the associated matroid, as defined by Feichtner and Yuzvinsky \cite{FY}. In this setting, Adiprasito, Huh and Katz proved that the Chow ring of a matroid satisfies a combinatorial analogue of the Kähler package \cite{chowring}. This collection of properties includes Poincaré duality, strong Lefschetz property and the Hodge--Riemann relations, which imply the palindromicity and unimodality of the corresponding Chow polynomials. For general posets there is not such a ring-theoretic interpretation, however similar properties still hold. In particular, it is known that the coefficients of the characteristic Chow polynomial of any poset are non-negative, palindromic, and unimodal, and that they are also $\gamma$-positive for Cohen--Macaulay posets \cite{ferroni-matherne-vecchi}. 
An outstanding open problem is whether a ring with Poincaré duality and the strong Lefschetz property can be constructed for any poset so that its Hilbert--Poincaré series coincides with the Chow polynomial \cite[Section 4.6]{ferroni-matherne-vecchi}.

In this paper, we will show that there exist algebras with these properties for any weakly ranked poset. This serves as further evidence for the possibility of finding a suitable definition for the generalization of Chow rings. 
Let us denote by $h = h(P) = (h_0,h_1, \ldots,h_{n-1})$ the sequence of coefficients of the Chow polynomial of a given poset $P$ of weak rank $n$. 
Our first main result is the following.
\begin{theorem}\label{thm:main1}
    Let $P$ be a weakly ranked, 
    finite, bounded poset of weak rank $n$. Then $h(P)$ forms a $\SI$-sequence. 
\end{theorem}
See Definitions \ref{def:O} and \ref{def:SI} for the definition of $\SI$-sequence. This statement is equivalent to saying that $h(P)$ is the $h$-vector of a simplicial polytope, see also Theorem \ref{thm:characterization SI sequence} below. The proof proceeds by giving an explicit combinatorial construction of a monomial order ideal whose graded pieces realize the successive differences of the Chow coefficients. A key ingredient is an analysis of Feichtner–-Yuzvinsky monomials, combined with symmetric chain decompositions of products of chains. 
Theorem \ref{thm:main1} allows us to state the following corollary immediately.

\begin{corollary}\label{cor:main lefschetz}
    Let $P$ be a weakly ranked, finite, bounded poset of weak rank $n$. Then there exists a Gorenstein algebra with Kähler package whose Hilbert--Poincaré series coincides with the Chow polynomial of $P$.
\end{corollary}
 
Furthermore, we are able to show that if the poset is ranked, this monomial ideal is pure, i.e., all its maximal monomials are of the same degree. This is the second main result of our paper.
\begin{theorem}\label{thm:main2}
    Let $P$ be a ranked, finite, bounded poset of rank $n$. Then the sequence
    \[
    \left(h_0,h_1 - h_0, \ldots,h_{\left\lfloor\frac{n-1}{2}\right\rfloor} - h_{\left\lfloor\frac{n-1}{2}\right\rfloor - 1}\right)
    \]
    is a pure $O$-sequence.
\end{theorem}

As an application, we study the log-concavity of Chow polynomials. Using general properties of $O$-sequences, we show that Chow polynomials are log-concave for all posets of weak rank at most six. We also construct explicit counterexamples in every higher weak rank.

The paper is organized as follows. In Section \ref{sec:background} we recall the necessary background on standard graded algebras, Chow polynomials, monomial order ideals and $O$-sequences and $\SI$-sequences. In Section \ref{sec:symmetric chain decomp} we recall and prove relevant results about symmetric chain decompositions. Section \ref{sec:main1} contains the proof of Theorem \ref{thm:main1} and hence Corollary \ref{cor:main lefschetz}. Section \ref{sec:main2} proves Theorem \ref{thm:main2}. Lastly, Section \ref{sec:log-concavity} discusses the log-concavity of Chow polynomials. 

\section*{Acknowldgments}
The authors wish to thank Petter Brändén and Luis Ferroni for useful discussions.

\section{Background}\label{sec:background} In this section we quickly review the necessary tools that we need to prove our results.
\subsection{Properties of sequences}
Given a finite sequence of integers $(a_0,a_1, \ldots, a_d)$ or, equivalently, the polynomial $p(t) = \sum_{i=0}^{d}a_it^i$, we say that it is
\begin{itemize}
    \item \emph{palindromic with center of symmetry $d/2$} if $a_i = a_{d-i}$ for every $0\leq i \leq d$,
    \item \emph{unimodal} if there exists an index $j$ such that 
    \[
    a_0 \leq a_1 \leq \ldots \leq a_{j-1} \leq a_j \geq a_{j+1} \geq \ldots a_d,
    \]
    \item \emph{log-concave} if $a_i^2 \geq a_{i-1}a_{i+1}$ for every $1 \leq i \leq d-1$,
    \item \emph{$\gamma$-positive} if $p(t)$ is palindromic and we can write
    \[
    p(t) = \sum_{i=0}^{\lfloor d/2\rfloor}\gamma_i t^i(1+t)^{d-2i}
    \]
    with all the coefficients $\gamma_i \geq 0$,
    \item \emph{real-rooted} if all the zeros of $p(t)$ are real.
\end{itemize}

For two sequences corresponding to polynomials $p,q$ we refer to the sequence of coefficients of their product $pq$ as their \emph{convolution}.
It is well known that the following chains of (strict) implications hold:
\[
\text{real-rooted} \implies \text{log-concave} \implies\text{unimodal}
\]
and, when $p(t)$ is palindromic,
\[
\text{real-rooted} \implies \text{$\gamma$-positive} \implies\text{unimodal}.
\]

\subsection{Standard graded algebras}

A graded algebra $$A=A_0\oplus A_1 \oplus \ldots$$
over a ring $R$ is said to be a \emph{standard graded algebra} if $A_0\simeq R$ and $A$ is generated by the elements of $A_1$ as an algebra.
In this paper, we will consider graded Artinian algebras over fields, that is, we assume that we may express $A$ as $A=K[x_1,x_2,\ldots,x_d]/I$ and that $$A=A_0\oplus A_1 \oplus \ldots \oplus A_e.$$
We refer to the series $$(\dim_K(A_0),\dim_K(A_1)\ldots,\dim_K(A_e))$$ 
or, equivalently, to the polynomial $H_A(t) \coloneq \sum_{i=0}^e \dim_K(A_i)t^i$ as the \emph{Hilbert--Poincaré series of $A$}. Let us introduce a few important concepts in the study of such algebras. 

\begin{definition}
    Let $A= \bigoplus_{i=0}^e A_i$ be a finite-dimensional standard graded Artinian algebra with $\dim_K(A_e) = 1$ and let $d: A_e\to K$ be an isomorphism. We say that $(A,d)$ has \emph{Poincaré duality} if the bilinear maps $$A_k\times A_{e-k}\to K \qquad (x,y) \mapsto d(xy)$$ are all non-degenerate.
\end{definition}

An immediate consequence of Poincaré duality is that the Hilbert--Poincaré series is palindromic.
A standard graded Artinian algebra over a field is called \emph{Gorenstein} if and only if it exhibits Poincaré duality for a choice of $d$ \cite[Section 1B]{gorensteinpdduality}.

\begin{definition}
    An Artinian Gorenstein algebra $A=\bigoplus_{i=0}^e A_i$ satisfies the \emph{Weak Lefschetz property} if there exists an element $\ell \in A_1$ such that the map $$\times \ell : A_k\to A_{k+1}$$ 
    given by the multiplication by $\ell$ is injective for $k\leq \lfloor (e-1)/2\rfloor$ and surjective for $k\geq \lfloor e/2\rfloor$. 
    We say that $A$ satisfies the \emph{Strong Lefschetz property} if there exists an element $\ell \in A_1$ such that the map $$\times \ell^{e-2k}: A_k\to A_{e-k}$$ is an isomorphism for every $k$. We refer to such an element $\ell$ as a \emph{Lefschetz element}.   
\end{definition}
Note that the Strong Lefschetz property implies the Weak Lefschetz property, thus implying unimodality for the Hilbert--Poincaré series. For a more thorough introduction to the Weak and Strong Lefschetz property, we refer the reader to \cite{Lefschetz}.

\begin{definition}
    Let $A=\bigoplus_{i=0}^e A_i$ be an Artinian Gorenstein algebra over $\bR$ with a Lefschetz element $\ell$. We say that $A$ satisfies the \emph{Hodge--Riemann relations} if the bilinear maps defined as $$A_k\times A_k\to \bR \quad (x,y) \mapsto  (-1)^k d(xy\ell^{e-2k}),$$
    are positive definite when restricted to ${\ker(\ell^{e-2k+1})\subset A_k}$.  
\end{definition}

An Artinian Gorenstein algebra over $\bR$ has the \emph{Kähler package}, if it satisfies Poincaré duality, the Strong Lefschetz property, and the Hodge-Riemann relations.

\subsection{Chow polynomials}
In this section, we briefly recall the definition of our main object of interest, i.e., the \emph{Chow polynomial} of a \emph{partially ordered set} (or poset for short). Recall that an interval of a poset is a subposet of $P$ of the form $[x,y]= \{z \in P : x\leq z \leq y\}$, where $x,y\in P$. 
A poset $P$ is \emph{locally finite} if each of its intervals is finite. We also say that a poset is \emph{bounded} if it has unique least and largest elements, which we denote by $\zero$ and $\one$, respectively.
Given a poset $P$, a \emph{weak rank function} is a function $\rho: P\times P \to \bN$ such that 
\begin{itemize}
    \item $\rho(x,y) = \rho_{x,y} > 0$ if and only if $x < y$, and
    \item  $\rho_{x,y} = \rho_{x,z} + \rho_{z,y}$, for all  $x\leq z \leq y$ in $P$.
\end{itemize}
A \emph{weakly ranked} poset consists of a pair $(P,\rho)$, where $\rho$ is a weak rank function for $P$. By slight abuse of notation, we say that $P$ is a weakly ranked poset, when this does not create confusion. 
If the poset has a least element $\zero$, we write $\rho(x):= \rho_{\zero,x}$ for an element $x \in P$ and call this the \emph{weak rank} of $x$ in $P$. If $\rho_{x,y} =1$ whenever $y$ covers $x$, then we say that $P$ is \emph{ranked}. Moreover, if $P$ is bounded, then we say that $P$ has weak rank $\rho(\one)$. Notice that if a poset is bounded and locally finite, then in particular it is finite. All the posets considered in this paper are finite, bounded and weakly ranked.
Lastly, we will denote by $W(P) = (W_0,W_1,\ldots W_{\rho(\one)})$, where $W_i=|\{p\in P | \rho(p)=i\}|,$ the sequence of \emph{Whitney numbers of the first kind of $P$}.

While the following is not the standard definition of Chow polynomials, it is the one which is going to be more convenient for us. 
\begin{definition}[{\cite[Theorem~4.1]{ferroni-matherne-vecchi}}]\label{def:chow polynomial}
    Let $P$ be a weakly ranked poset. Then its Chow polynomial is
    \[
    H_P(t) = \sum_{s \geq 0}\sum_{\zero=p_0 < p_1 < \cdots < p_s \leq \one} \prod_{i=1}^{s}\frac{t\left(t^{\rho(p_i)-\rho(p_{i-1})-1}-1 \right)}{t-1} .
    \]
\end{definition}
\begin{remark}
The polynomial $H_P$ is usually called \emph{characteristic} Chow polynomial, as it is proven to be the inverse of the negative reduced characteristic function inside the incidence algebra of the poset $P$. Indeed, the original construction in \cite{ferroni-matherne-vecchi} associates a Chow polynomial to any poset with a choice of extra data called \emph{$P$-kernel}. We also observe that this definition would associate a polynomial to every interval in a poset, that would not need to be finite (only locally finite) nor bounded. However, since the characteristic Chow polynomial is invariant under isomorphism of posets, we can reduce ourselves to studying bounded, finite posets. Since in this article we are only concerned with the characteristic Chow polynomial, we choose to ease the notation and only refer to it as \emph{Chow polynomial}.
\end{remark}

The following is a summary of the properties that Chow polynomials are proven to satisfy {\cite{ferroni-matherne-vecchi}}.
\begin{theorem}\label{thm:properties of characteristic Chow}
    For any poset $P$, the Chow polynomial is a monic polynomial of degree $\rho(P)-1$ with non-negative, palindromic and unimodal coefficients. Moreover, if $P$ is Cohen--Macaulay, then it is also $\gamma$-positive. 
\end{theorem}

Chow polynomials are actually conjectured to be real-rooted for every Cohen--Macaulay poset \cite[Conjecture~1.5]{ferroni-matherne-vecchi}. The conjecture has recently gained a lot of attention and was proven for different classes of posets in \cite{uniformrealrooted,hoster-stump,TNrealrooted,coron-li-ferroni}.

In \cite{FY}, Feichtner and Yuzvinsky defined a toric variety associated to any atomistic lattice and studied its \emph{Chow ring}, giving an explicit presentation and a Gröbner basis for its ideal of relations. Later, Adiprasito, Huh and Katz showed that when $P$ is a geometric lattice the corresponding \emph{Chow ring} is an Artinian Gorenstein algebra with the Kähler package \cite{chowring}. These properties are easily seen to not hold for the Feichtner--Yuzvinsky Chow ring of generic atomistic lattices.

By comparing the Feichtner--Yuzvinsky Gröbner basis with Definition \ref{def:chow polynomial}, it is clear that when $P$ is a geometric lattice its Chow polynomial coincides with the Hilbert--Poincaré series of its Chow ring, see also \cite[Theorem 1.5]{ferroni-matherne-stevens-vecchi}. This, amongst other indicative numerical properties of the Chow polynomial led Ferroni, Matherne and the second author to conjecture in \cite[Section 4.6]{ferroni-matherne-vecchi} the existence of a combinatorial ring generalizing the Chow ring for arbitrary weakly ranked posets exhibiting various strong properties, including Poincaré duality and the Strong Lefschetz property.

\subsection{Monomial order ideals and (pure) \texorpdfstring{$O$}{O}-sequences} 
    A finite set of monomials $M$ is called a \emph{monomial order ideal} if any monomial that divides a monomial in $M$ is also contained in $M$. A monomial order ideal $M$ is \emph{pure} of degree $d$ if any maximal monomial in $M$ has degree $d$. For a monomial order ideal $M$, let $h_d$ denote the number of monomials of degree $d$. We refer to the vector $(h_0,h_1,h_2,\ldots,h_e)$ as the \emph{$h$-vector} of $M$. Note that this is exactly the sequence of Whitney numbers of the first kind of $M$ as a poset ordered by degree.
    \begin{definition}\label{def:O}
        We say that a sequence of nonnegative integers $$h=(h_0,h_1,h_2,\ldots,h_e)$$ is a \textit{(pure) $O$-sequence} if it is the $h$-vector of a (pure) monomial order ideal.
    \end{definition}
    We refer to \cite{Migliore2013} for a more complete introduction to the concept.
    \begin{definition}
        Let $P$ be a poset of rank $n$ and let $H_P(t) = \sum_{i=0}^{n-1}h_it^i$ be its Chow polynomial. We call the sequence
        \[
        h = h(P) = (h_0, h_1, \ldots, h_{n-1})
        \]
        the \emph{$\FY$ $h$-vector of $P$}.
    \end{definition}
    We justify this name by constructing a monomial order ideal with $h(P)$ as its $h$-vector.    
    Consider the polynomial ring $\bQ[x_p \mid p \in P \setminus \zero]$ and all monomials $m$ of the form
    \[
        m = x_{p_1}^{\ell_1}x_{p_2}^{\ell_2}\cdots x_{p_s}^{\ell_s},
    \]
    where ${\zero= p_0 < p_1 < \cdots < p_s \leq \one}$ is a chain of $P$, $1\leq \ell_i \leq d_i-1$ and $d_i = \rho(p_i) - \rho(p_{i-1})$.
    We say that a monomial of this form is a \emph{Feichtner--Yuzvinsky monomial} (or FY monomial for short). We denote by $\FY$ the family of all FY monomials and by $\FY^k$ the family of FY monomials of degree $k$. Clearly, $\FY$ is a, not necessarily pure, monomial order ideal. 
    The following is a direct consequence of the definition of $H$ (the argument in \cite[Corollary~1 and p. 526]{FY} works in general for posets).
    \begin{proposition}\label{prop:FYhvector}
        For any poset $P$,
        \[
            h_k = \left | \FY^k \right|.
        \]
    \end{proposition}
    
    The following result highlights the importance that $O$-sequences have in commutative algebra.

    \begin{theorem}[{\cite[Theorem~2.2]{Osequence1}}]
        The sequence $(1,h_1,h_2,\ldots,h_e)$ is an $O$-sequence if and only if it is the Hilbert--Poincaré series of a standard graded algebra.
    \end{theorem}

    We have a complete characterization of $O$-sequences. For any pair of nonnegative integers $n$ and $d$ there exists a unique sequence of numbers $k_d>k_{d-1}>\ldots >k_\delta\geq \delta\geq 1 $ such that $$n=\binom{k_d}{d}+\binom{k_{d-1}}{d-1}+\ldots+\binom{k_\delta}{\delta}.$$ This is referred to as the \textit{$d$-binomial expansion of $n$}
    
    \begin{theorem}[{\cite{MacaulaySomePO}}]\label{thm:Odescription}
        A sequence $(1,h_1,h_2,\ldots,h_e)$ is an $O$-sequence if and only if for every $1\leq i < e$
        $$h_i=\binom{k_d}{d}+\binom{k_{d-1}}{d-1}+\ldots+\binom{k_\delta}{\delta}$$ and
        $$h_{i+1}\leq \binom{k_d+1}{d+1}+\binom{k_{d-1}+1}{d}+\ldots+\binom{k_\delta+1}{\delta+1}.$$
    \end{theorem}

    While this statement gives a complete description of $O$-sequences, the same has not yet been achieved for pure $O$-sequences.
    Nonetheless, pure $O$-sequences of length $3$ can be characterized as follows.
    \begin{proposition}\label{prop:pure3seq}
        A sequence $(1,h_1,h_2)$ is a pure $O$-sequence if and only if $$\left\lfloor \frac{h_1+1}{2} \right\rfloor\leq  h_2\leq \binom{h_1+1}{2}.$$
    \end{proposition}
    \begin{proof}
    The second inequality is equivalent to the inequality in Theorem \ref{thm:Odescription} that characterizes $O$-sequences, and hence is a necessary condition.
    If $M$ is a pure monomial order ideal with $h$-vector $(1,h_1,h_2)$, every linear monomial must be a divisor of a degree two monomial. As every degree two monomial is divisible by at most two different linear monomial, the first inequality is necessary. 

    Conversely, for any vector $(1,h_1,h_2)$ satisfying these properties we may take $h_1$ different variables $x_1,x_2,\ldots, x_{h_1}$ and choose $h_2$ degree two monomials to get a pure monomial order ideal. We start by choosing $\left\{x_{2i-1}x_{2i}\right\}_{i=1}^{\left\lfloor h_1/2 \right\rfloor}$ and then we add to the set the monomial $x_{h_1}^2$ if $h_1$ is odd. This subset has size exactly $\left\lfloor (h_1+1)/2\right\rfloor$ elements, and covers each of the $h_1$ degree one monomials. Thus, any choice for the remaining $h_2-\left\lfloor(h_1+1)/2\right\rfloor$ monomials gives a pure monomial order ideal. \end{proof}

    While we may not have a complete characterization for pure $O$-sequences, we still have strong inequalities holding for them.
    \begin{theorem}[\cite{Hibi1989WhatCB}]
        For a pure $O$-sequence $(1,h_1,h_2,\ldots,h_e)$ we have that for any $i\leq j\leq e-i$ $$h_i\leq h_j.$$
    \end{theorem}
    \begin{corollary}
    For a pure $O$-sequence
    \begin{itemize}
        \item $h_i\leq h_{e-i}$ for any $i\leq e/2$,
        \item $h_0\leq h_1\leq \cdots \leq h_{\lfloor e/2\rfloor}$.
    \end{itemize}
    \end{corollary}

\subsection{\texorpdfstring{$\SI$}{SI}-sequences}    
    For a palindromic and unimodal sequence $h=(1,h_1,\ldots,h_e)$ we define its \textit{differential sequence} to be $$\Delta h=(1,h_1-1,h_2-h_1,\ldots ,h_{\lfloor e/2\rfloor}-h_{\lfloor e/2\rfloor-1}).$$

    Sommerville proved in \cite{DehnSommerville} that the \emph{$h$-vector} of a simplicial polytope $P$ is unimodal and palindromic. McMullen \cite{g-conjecture} then conjectured that a vector is the $h$-vector of a simplicial polytope if and only if $\Delta h$, also known as the \emph{$g$-vector}, is an $O$-sequence. This is commonly known as the \emph{$g$-conjecture} for polytopes.

    Later, McMullen defined in \cite{gThmNecessity2} the \emph{polytope algebra} $\Pi_P$ for any simplicial polytope $P$. This is a standard graded algebra with the Kähler package whose Hilbert series is equal to the $h$-vector of $P$. When taking a quotient of $\Pi_P$ by an ideal generated by any Lefschetz element, one obtains a standard graded algebra whose Hilbert--Poincaré series coincides with the $g$-vector, thus giving a proof for the necessity of the statement. 

    \begin{definition}\label{def:SI}
        A \emph{Stanley-Iarrobino sequence} (or \emph{$\SI$-sequence} for short) is a nonnegative, palindromic sequence $h$ that is a \emph{differentially $O$-sequence}, i.e., $\Delta h$ is an $O$-sequence.
    \end{definition}
    
    \begin{theorem}[\cite{gThmSufficiency,gThmNecessity, gThmNecessity2,Lefschetz}]\label{thm:characterization SI sequence}
        For a sequence $$h=(1,h_1,\ldots,h_e)$$ the following statements are equivalent:
        \begin{itemize}
            \item $h$ is an $SI$-sequence,
            \item $h$ is the $h$-vector of a simplicial polytope,
            \item $h$ is the Hilbert--Poincaré series of a Gorenstein algebra with the Weak Lefschetz property,
            \item $h$ is the Hilbert--Poincaré series of a Gorenstein algebra with the Strong Lefschetz property,
            \item $h$ is the Hilbert--Poincaré series of a Gorenstein algebra with the Kähler package.
        \end{itemize}
    \end{theorem}
    \begin{remark}
        In \cite[Problem 2.9]{MR4780720}, Brenti asks whether the coefficients of every monic, palindromic and real-rooted polynomial form a $\SI$-sequence. Solving the real-rootedness conjecture for Chow polynomials and Brenti's conjecture would provide an independent proof of Theorem \ref{thm:main1} for Cohen--Macaulay posets. We observe however that our result holds for a much larger class of posets.  
    \end{remark}
    Our goal now is to prove that the $\FY$ $h$-vector is a $\SI$-sequence. To do so, in analogy with the polytope case, we define the $\FY$ $g$-vector of $P$ as $g = g(P) \coloneq \Delta h(P)$. By Theorem \ref{thm:properties of characteristic Chow} the only thing we would need to show then is that $g$ is an $O$-sequence.
    
\section{Symmetric chain decompositions}\label{sec:symmetric chain decomp}
The goal of this section is to find a combinatorial interpretation of the $\FY$ $g$-vector of a poset.
A \emph{symmetric chain decomposition} of a poset $P$ of rank $n$ is a partition of $P$ such that each part is a saturated chain, i.e., no element can be added between two consecutive elements without losing the property of being a chain, where if the first element is of rank $r$, the last element is of rank $n-r$. It is known that if $P_1$ and $P_2$ are two posets admitting a symmetric chain decomposition, then so is the Cartesian product $P_1\times P_2$ \cite{chainDecomp}. Let $C_n$ denote the chain poset of rank $n$.
\begin{corollary}\label{cor:scd for product of chains}
    The poset $P = C_1 \times C_2 \times \cdots \times C_s$ admits a symmetric chain decomposition. 
\end{corollary}

For a symmetric chain decomposition $S$ we will refer to the set of minimal elements of each saturated symmetric chain as the \emph{set of initial elements of $S$}, denoted as $S_{\text{init}}$.

\begin{proposition}
For a ranked poset $P$ of rank $n$ with Whitney numbers of the first kind given by the sequence $W = (W_0,W_1,\ldots, W_n)$ and a symmetric chain decomposition $S$, 
$$W(S_\text{init})=\Delta W.$$
\end{proposition}
\begin{proof}
    Every element of rank 0 must be the start of a chain in the decomposition, thus, the number of initial elements with rank $0$ is exactly $W_0$.

    For $1 \leq s \leq \lfloor \frac{n-1}{2} \rfloor$, the number of distinct chains with an element of rank $s-1$ is exactly $W_{s-1}$ (as every such element is contained in exactly one chain).
    By the symmetry of the chains, we know that none of these chains end at $s-1$, as $n-s+1>s-1$, thus, as the chains are saturated each of these chains must contain an element of rank $s$. 
    The remaining $W_s-W_{s-1}$ elements of rank $s$ must be the initial element of a chain, as each chain in $S$ is saturated.
    Lastly, no chain can start with an element larger than $\lfloor \frac{n-1}{2} \rfloor$, thus for any larger $s$ the left-hand side is zero.
\end{proof}

The following observation will help us construct symmetric chain decompositions.
\begin{remark}\label{rem:subdiv-decomp}
    Let $P$ be a ranked poset of rank $n$. If we have a partition $P = \bigsqcup_i P_i$ into ranked subposets, each with a sequence of Whitney numbers of the first kind palindromic with center of symmetry $n/2$, then the union of any symmetric chain decomposition of $P_i$ constitutes a symmetric chain decomposition of $P$. 
\end{remark}

We will now construct a specific symmetric chain decomposition for a product of chains, which exists by Corollary \ref{cor:scd for product of chains}, and give an explicit description of its set of initial elements.

\begin{figure}
        \centering
        \begin{tikzpicture}[scale=0.8]
          \def\r1{6}   
          \def\L{4}    \draw[step=1,gray!30,very thin] (0,0) grid (\r1,\L);
        
          \draw (0,0) -- (\r1,0) node[midway,yshift=-17pt] {$C_{r_1}$};
          \draw (0,0) -- (0,\L) node[midway, xshift=-17pt] {$S'_i$};
        
          \foreach \x in {0,...,\r1}
            \node[below] at (\x,0) {\small $\x$};
          \foreach \y in {0,...,\L}
            \node[left] at (0,\y) {\small $\y$};
        
          \foreach \i in {0,...,6} { 
            \ifnum\i>\L\relax\else
              \draw[line width=1.2pt]
                (\i,0) -- (\i,{\L-\i}) -- (\r1,{\L-\i});
        
              \foreach \y in {0,...,\L} {
                  \ifnum\y>\L-\i\relax\else
                    \fill (\i,\y) circle (1.7pt);
                    \ifnum\y=0
                    \fill (\i,\y) circle (2.7pt);
                    
                    \fi
                    
                  \fi
              }
              \foreach \x in {\i,...,\r1} {
                \fill (\x,{\L-\i}) circle (1.7pt);
              }
            \fi
          }
        
        \end{tikzpicture}
        \caption{Symmetric chain decomposition of the product $C_{r_1}\times S'_i$}
        \label{fig:prod-decomp}
    \end{figure}
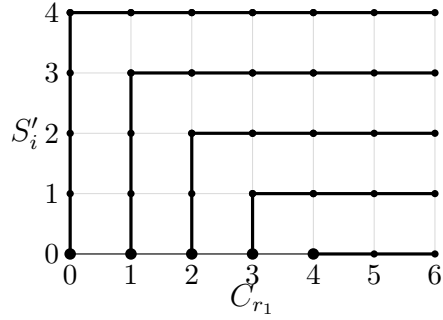
\begin{theorem}\label{thm:chain-decomp}
    For any poset $P$ of the form $$P=C_{r_1}\times C_{r_2}\times \cdots \times C_{r_n}$$ there exists a symmetric chain decomposition $S$ such that \begin{equation}\label{eq:S init}
    S_{\text{init}}=\big\{(a_1,\ldots,a_n) \ | \ a_s\leq r_s, \; a_s\leq \sum_{i=s+1}^{n} r_i-2a_i \big\}.
    \end{equation}
\end{theorem}
\begin{proof}
    We are going to prove this statement by induction on $n$.
    If $n=1$, then $S = \{C_{r_1}\}$ and the only initial element is $(0)$.
    For any $n>1$ let us construct the decomposition as follows. By induction, we consider the symmetric chain decomposition $S'$ of $$P'=C_{r_2}\times \ldots \times C_{r_n}=\bigsqcup_i S'_i$$ that satisfies \eqref{eq:S init}.
    Let us take the partition 
    $$P=C_{r_1}\times P'=\bigsqcup_i C_{r_1}\times S'_i.$$ 
    We will give a symmetric chain decomposition of each component ${C_{r_1}\times S'_i}$; as these subsets are symmetric in $P$ this gives a symmetric chain decomposition of $P$ as discussed in Remark \ref{rem:subdiv-decomp}. 
    The initial elements $s'_i$ of $S'_i$ are exactly all possible vectors $(a_2,a_3,\ldots,a_n)$ that satisfy \eqref{eq:S init} for $2\leq s \leq n$.
    Notice also that their ranks are $\rho(s'_i)=\sum_{j=2}^n a_j$ and the length of $S'_i$ is exactly 
    \[
    \len(S'_i)=\rho(P')-2\rho(s'_i)=\sum_{j=2}^n r_j - 2 a_j
    \] 
    by symmetry.
    Now consider the symmetric chain decomposition given by $S_{j,i}$ chosen as
    \[
    (j,0) \prec (j,1) \prec \ldots \prec (j, \len(S_i') - j) \prec (j+1, \len(S_i') - j) \prec \ldots \prec (r_1, \len(S_i') - j),
    \]
    with $0 \leq j \leq \min(r_1,\len(S'_i))$. 
    This decomposition is depicted in Figure~\ref{fig:prod-decomp}.
    
    The initial elements in this decomposition are exactly the ones where $(a_2,\ldots, a_n)$ satisfies \eqref{eq:S init} for $P'$ and ${0 \leq a_1 \leq \min(r_1,\sum_{j=2}^n r_j-2a_j)}$.
    Taking the union of these decompositions over all $S'_i$ we get exactly the set of initial elements described in the statement. 
\end{proof}

We will now use this decomposition to give an explicit description of a well-chosen symmetric chain decomposition of $\FY$. Similar techniques were employed by Angarone, Nathanson and Reiner in \cite{angarone-nathanson-reiner}. Later, we will prove that the initial elements of this decomposition constitute a monomial order ideal, thus, proving Theorem \ref{thm:main1}.

\begin{theorem}\label{thm:description SFY}
    For any weakly ranked poset $P$ there exists a symmetric chain decomposition $S$ of $\FY$ such that the set of initial elements, which we denote by $\SFY$, coincides with the family of FY monomials of the form
        \[
        m = \prod_{k=1}^s x_{p_k}^{\ell_k},
        \]
        where $\zero = p_0 < p_1 < \cdots < p_s < \one$ is a chain of the poset $P$ and 
        \[
        1 \leq \ell_k \leq  \min\left(\rho(p_k) - \rho(p_{k-1}) - 1,\ \rho(P) - \rho(p_k) - 2\sum_{i= k+1}^{s}\ell_i\right).
        \]
\end{theorem}
\begin{proof}
    Let us build a symmetric chain decomposition of $\FY$.
    
    First, for a chain $\cC = \{\zero = p_0 < p_1<p_2<\cdots<p_s<\one\}$ we define $\FY_\cC$ to be the subposet of $\FY$ consisting of monomials with support either $\cC$ or $\cC \cup \{\one\}$. The bottom element of $\FY_\cC$ is $\prod_{i} x_{p_i}$ with rank $s$, while its top element, $\prod_i x_{p_i}^{d_i}\ x_{\one}^{d_{\one}}$ has rank
    \[
    \sum_{i=1}^s (d_i - 1) + d_{\one}- 1= \rho(P) - s - 1. 
    \]
    Hence, $\FY_\cC$ is symmetric inside $\FY$. Moreover, since $\FY_\cC$ is isomorphic to the product of chains 
    $$C_{d_1-1}\times C_{d_2-1}\times \ldots \times C_{d_s-1} \times C_{d_{\one}},
    $$ 
    by Theorem \ref{thm:chain-decomp}, $\FY_\cC$ admits a symmetric chain decomposition $S$.
    We now describe the elements in $S_\text{init}$ using \ref{thm:chain-decomp}, where the vector $$(a_1,a_2,\ldots,a_s,a_{\one})$$ will correspond to $x_{p_1}^{a_1+1}x_{p_2}^{a_2+1}\ldots x_{p_s}^{a_s+1}x_{\one}^{a_{\one}}$. 
    First, let us note that the exponent of $x_{\one}$ will always be 0.
    Thus, the set may be described as
    $$\big\{x_{p_1}^{a_1+1}x_{p_2}^{a_2+1}\ldots x_{p_s}^{a_s+1} | a_i\leq d_i-1, a_k\leq d_{\one} + \sum_{i=k+1}^{s} (d_i-1)-2a_i \big\}.$$

    The theorem is proven by setting $\ell_i:=1+a_i$ after observing that
    \begin{align*}
        1 + (d_{\one}-1) + \sum_{i=k+1}^s(d_i-1 - 1) - 2a_i &= d_{\one} + \sum_{i=k+1}^s(d_i - 2\ell_i) \\
        &= \rho(P) - \rho(p_k) - 2 \sum_{i=k+1}^s\ell_i.
    \end{align*}
\end{proof}

\begin{remark}
    We observe that when we apply Theorem \ref{thm:chain-decomp}, the resulting symmetric chain decomposition depends on an explicit choice of ordering the terms in the cartesian product. The choice we make in the proof above is needed for the set $\SFY$ to be a monomial order ideal, which is used in the proof of Theorem \ref{thm:main1} below.
\end{remark}

\section{Proof of Theorem \ref{thm:main1}}\label{sec:main1}
    We are now ready to prove Theorem \ref{thm:main1}.
    According to Theorem \ref{thm:properties of characteristic Chow}, we only need to prove that $h$ forms a differential $O$-sequence. We will do so by constructing directly a monomial order ideal. The proof of the theorem will follow immediately from the following result.
    \begin{theorem}\label{thm:SFY monomial order ideal}
        Let $P$ be a poset.
        The set of monomials $\SFY$ is a monomial order ideal.
    \end{theorem}
    \begin{proof}
        To prove the statement, it is sufficient to show that for any monomial
        $$m = \prod_{k=1}^s x_{p_k}^{\ell_k}\in \SFY,$$ any monomial $m'=\prod_{k=1}^s x_{p_k}^{\ell'_k}$ that divides $m$ is also in $\SFY$.
        By the definition of $\SFY$ we know that \begin{equation}\label{eq:mon-1}
            \ell_k\leq\rho(p_k)-\rho(p_{k-1})-1
        \end{equation} and that \begin{equation}\label{eq:mon-2}
            \ell_k\leq \rho(\one)-\rho(p_k)-2\sum_{i=k+1}^{s} \ell_i.
        \end{equation}

        To show that $m'$ belongs to $\SFY$ we need to show that every $\ell'_k$ satisfies the same inequalities (written for the support of $m'$).
        As $\ell'\leq \ell$, we know that 
        $$\ell'_k \leq \ell_k \leq \rho(\one)-\rho(p_k)-2\sum_{i=k+1}^{s} \ell_i \leq \rho(\one)-\rho(p_k)-2\sum_{i=k+1}^{s} \ell'_i,$$
        thus \eqref{eq:mon-2} holds for the exponents of $m'$.
        We move on to \eqref{eq:mon-1}. 
        For a fixed $k$, let us assume that $p_s$ is the largest element below $p_k$ such that $\ell'_s>0$, or $\zero$ if no such element exists. Now, we need to show that $$\ell'_k\leq \rho(p_{k})-\rho(p_s)-1.$$
        Since $p_s\leq p_{k-1}<p_k$, we can write $$\ell'_k\leq \ell_k\leq \rho(p_k)-\rho(p_{k-1})-1\leq \rho(p_k)-\rho(p_s)-1.$$
         
    \end{proof}

\begin{proof}[Proof of Theorem \ref{thm:main1} and thus of Corollary \ref{cor:main lefschetz}]
    The fact that $h$ forms a differential $O$-sequence follows from Theorem \ref{thm:description SFY} and Theorem \ref{thm:SFY monomial order ideal} as they imply that the $\FY$ $g$-vector is the $h$-vector of the monomial order ideal $\SFY$. The corollary follows by Theorem \ref{thm:characterization SI sequence}.
\end{proof}

\section{Proof of Theorem \ref{thm:main2}}\label{sec:main2}
    In this section, we prove Theorem \ref{thm:main2}.
    \begin{proof}[Proof of Theorem \ref{thm:main2}]
        We show that any monomial that is maximal under divisibility has the correct degree, i.e., $\left\lfloor(n-1)/2\right\rfloor$. To do so, consider a maximal monomial $m$ and write $m = x_{p_k}^{\ell}\cdot m'$, where $m'$ is a monomial of degree $r$ with variables corresponding to elements that are strictly greater than the element $p_k$ of rank $k$. By contradiction, suppose that $\ell + r < \left\lfloor(n-1)/2\right\rfloor$. Clearly, $\ell = \min(n-k-2r, k-1)$, otherwise we could increase this exponent without violating the defining conditions, contradicting the maximality of $m$. We now distinguish two cases, 
        \begin{enumerate}
            \item[i)] $\ell = n-k-2r \leq k-1$, or
            \item[ii)] $\ell = k-1 < n-k-2r$.
        \end{enumerate}
        Let us start from case $i)$. Since $m$ is maximal, in particular we cannot add a variable $x_{p_2}$ corresponding to a rank $2$ element and still have a valid monomial. For us to be able to add $x_{p_2}$, we would need that $n-2-2(\ell + r) \geq 1$ and $k-2-1\geq \ell$. This means that at least one of the following must hold
        \begin{enumerate}
            \item[a)] $n-2-2(\ell + r) \leq 0$,
            \item[b)] $k-2-1 \leq \ell-1$. 
        \end{enumerate}
        If $a)$ holds, then $\ell + r \geq \left\lfloor\frac{n-1}{2}\right\rfloor$, implying that $m$ has the maximal possible degree. If $a)$ does not hold, then $b)$ does. This means that $r \leq \frac{n-2\ell-3}{2}$ and $\ell \geq k-2$. However, by $i)$ we know that $r=\frac{n-k-\ell}{2}$ and this would be satisfied if and only if  $\ell \leq k-3$. This concludes case $i)$.
        
        We move on to case $ii)$. Now, $m'$ can be seen as a monomial for the chain of rank $n-k$, so its maximal possible degree is $\left\lfloor(n-k-1)/2\right\rfloor$. If $r= \left\lfloor(n-k-1)/2\right\rfloor$, then $n-k-2r$ is either equal to $2$, if $n-k$ is even, or $1$, if $n-k$ is odd. However, since $\ell \geq 1$, then $n-k-2r \geq 2$. 
        To summarize,
        \[
            \ell = k-1 < n-k-2r = 2,
        \]
        which implies that $\ell = 1$, $k=2$ and $n$ is even. So, $$\ell + r = 1 + \left\lfloor\frac{n-2-1}{2}\right\rfloor = \left\lfloor\frac{n-1}{2}\right\rfloor.$$
        The last case to check is when $r < \left\lfloor(n-k-1)/2\right\rfloor$. If $r\leq \frac{n-2k-1}{2}$, then we could increase $r$ by one without violating any condition, contradicting the maximality of $m$. We can then assume that $n-k-2r-(k-1)=1$, which implies that $n$ is even and $r= \frac{n-2k}{2}$. Therefore,
        \[
        \ell + r = k-1 + \frac{n-2k}{2} = \frac{n-2}{2} = \left\lfloor \frac{n-1}{2} \right\rfloor.
        \]
        This concludes case $ii)$ and the proof. 
    \end{proof}

    Let us point out, that the theorem above relies on the fact that our poset is ranked (as opposed to being weakly ranked). We now show that the statement would not hold in the weakly ranked case. 
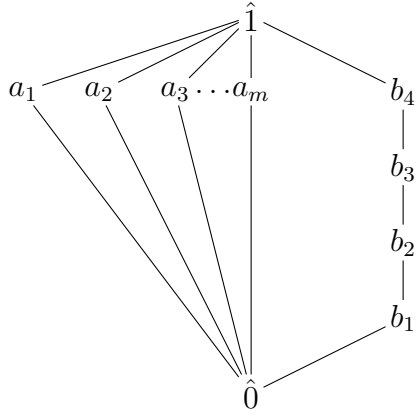
\begin{figure}
        \centering
        \begin{tikzpicture}[
            scale=1,
            every node/.style={ inner sep=1.5pt},
            level distance=1.2cm
        ]
        
        \node (0) at (0,0) {$\zero$};
        \node (a1) at (-3,4) {$a_1$};
        \node (a2) at (-2,4) {$a_2$};
        \node (a3) at (-1,4) {$a_3$};
        \node at (-0.5,4) {$\ldots$};
        \node (am) at (-0,4) {$a_m$};
        
        \node (b1) at (2,1) {$b_1$};
        
        \node (b2) at (2,2) {$b_{2}$};
        
        \node (b3) at (2,3) {$b_3$};
        \node (b4) at (2,4) {$b_4$};

        \node (1) at (0,5) {$\one$};
        
        \draw (0) -- (a1);
        \draw (0) -- (a2);
        \draw (0) -- (a3);
        \draw (0) -- (am);

        \draw (0) -- (b1);
        \draw (b1) -- (b2);
        \draw (b2) -- (b3);
        \draw (b3) -- (b4);
        
        \draw (a1) -- (1);
        \draw (a2) -- (1);
        \draw (a3) -- (1);
        \draw (am) -- (1);
        \draw (b4) -- (1);
        
        \end{tikzpicture}
        \caption{Counterexample to pureness}
        \label{fig:non-pure}
    \end{figure}

\begin{example}
    Let us consider the poset $P$ depicted in Figure \ref{fig:non-pure}. Let $\rho(a_i)=4$, $\rho(b_i)=i$ and $\rho(\one) = 5$.
        It is not hard to check that the $h$-vector of $\SFY$ in this case is $(1,m+3,2)$.
    Choosing a sufficiently large $m$ will imply that the inequality
    $$\left\lfloor\frac{h_1+1}{2}\right\rfloor \leq h_2 \leq \binom{h_1+1}{2}.$$
    from Proposition \ref{prop:pure3seq} cannot hold. Indeed, from our construction any $x_{a_i}$ and $x_{b_2}x_{b_4}$ would be in $\SFY$ and would be maximal.
\end{example}

\section{Log-concavity of the Chow polynomial}\label{sec:log-concavity}
In this section, we will discuss the log-concavity of Chow polynomials based on the weak rank of the poset.
We will use the following simple statement.
\begin{proposition}\label{prop:logdifference}
    For a symmetric and unimodal sequence $h$, if $\Delta h$ is log-concave, then $h$ is log-concave as well. 
\end{proposition}
\begin{proof}
    Let us prove that $h_i^2\geq h_{i-1}h_{i+1}$ for each index ${1\leq i\leq e-1}$. 
    Let us use the fact that the convolution of two log-concave sequences is log-concave \cite{logconcave}. Specifically, let us take the convolution of $(\Delta h_0,\Delta h_1,\ldots, \Delta h_{\lfloor (e-1)/2\rfloor})$ and the sequence $(1,1,\ldots,1)$ with length at least $\lfloor(e-1)/2\rfloor$. The first $\lfloor(e-1)/2\rfloor$ elements of this convolution are exactly $$h_0,h_1,\ldots h_{\lfloor(e-1)/2\rfloor}.$$ This implies the inequality for indices $1\leq i< \lfloor (e-1)/2\rfloor$.
    Since $h$ is palindromic, this also shows the inequality for $\lfloor e/2\rfloor< i\leq e-1$.
    Lastly, the inequality for $\lfloor (e-1)/2\rfloor$ and $\lfloor e/2\rfloor$ (these two might coincide) follows directly from unimodality. 
\end{proof}

Theorem \ref{thm:main1} allows us to get very strong bounds for the elements of the differential sequence of the coefficients of the Chow polynomial for low ranks. In particular, we are able to show log-concavity for posets of low rank.

\begin{theorem}
    Let $P$ be a poset of weak rank at most $6$. Then its Chow polynomial is log-concave.
\end{theorem}
    \begin{proof}
    The statement is trivial for weak rank at most $4$. If $P$ has weak rank $5$ or $6$, it has a $\FY$ $g$-vector of length $3$.
    By Proposition \ref{thm:Odescription} and since $g_1 \geq 1$, we can write 
    $$1\cdot g_2\leq g_1 \frac{(g_1+1)}{2} \leq g_1^2.$$ 
    The theorem now follows by Proposition \ref{prop:logdifference}. 
    \end{proof}
    \begin{remark}
        The only rank-exhaustive result we are aware of is \cite[Theorem 5.4]{ferroni-matherne-stevens-vecchi}, where real-rootedness (and therefore log-concavity) of Chow polynomials of geometric lattices of rank at most $5$ is proven exploiting the Koszulness of the Chow ring of the matroid. That kind of analysis cannot be extended yet to generic posets because there is no analogue of such a ring yet.
    \end{remark}
    We conclude this section by exhibiting a ranked poset with non log-concave Chow polynomial for any rank $\rho(P) > 6$.

    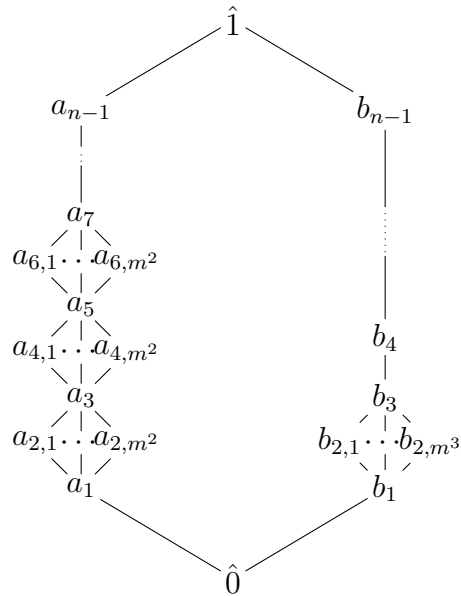
\begin{figure}
        \centering
        \begin{tikzpicture}[
            scale=1,
            every node/.style={ inner sep=1.5pt},
            level distance=1.2cm
        ]
        
        \node (0) at (2,0) {$\zero$};
        \node (a1) at (0,1.2) {$a_1$};
        
        \node (a21) at (-0.6,1.8) {$a_{2,1}$};
        \node (a22) at (0,1.8) {$\cdots$};
        \node (a23) at (0.6,1.8) {$a_{2,m^2}$};
        
        \node (a3) at (0,2.4) {$a_3$};
        
        \node (a41) at (-0.6,3.0) {$a_{4,1}$};
        \node (a42) at (0,3.0) {$\cdots$};
        \node (a43) at (0.6,3.0) {$a_{4,m^2}$};
        
        \node (a5) at (0,3.6) {$a_5$};
        
        \node (a61) at (-0.6,4.2) {$a_{6,1}$};
        \node (a62) at (0,4.2) {$\cdots$};
        \node (a63) at (0.6,4.2) {$a_{6,m^2}$};
        
        \node (a7) at (0,4.8) {$a_{7}$};
        
        \node[inner sep=0] (adots1) at (0,5.5) {};
        \node[inner sep=0] (adots2) at (0,5.7) {};
        \node (ar) at (0,6.2) {$a_{n-1}$};
        \node (1) at (2,7.4) {$\one$};
        
        \node (b1) at (4,1.2) {$b_1$};
        
        \node (b21) at (3.4,1.8) {$b_{2,1}$};
        \node (b22) at (4,1.8) {$\cdots$};
        \node (b23) at (4.6,1.8) {$b_{2,m^3}$};
        
        \node (b3) at (4,2.4) {$b_3$};
        \node (b4) at (4,3.2) {$b_4$};
        \node[inner sep=0] (bdots1) at (4,4.3) {};
        \node[inner sep=0] (bdots2) at (4,4.9) {};
        \node (br) at (4,6.2) {$b_{n-1}$};

        \draw (0) -- (a1);
        \draw (a1) -- (a21);
        \draw (a1) -- (a22);
        \draw (a1) -- (a23);
        
        \draw (a21) -- (a3);
        \draw (a22) -- (a3);
        \draw (a23) -- (a3);
        
        \draw (a3) -- (a41);
        \draw (a3) -- (a42);
        \draw (a3) -- (a43);
        
        \draw (a41) -- (a5);
        \draw (a42) -- (a5);
        \draw (a43) -- (a5);
        
        \draw (a5) -- (a61);
        \draw (a5) -- (a62);
        \draw (a5) -- (a63);
        
        \draw (a61) -- (a7);
        \draw (a62) -- (a7);
        \draw (a63) -- (a7);
        \draw (a7) -- (adots1);
        \draw[dotted] (adots1) -- (adots2);
        \draw (adots2) -- (ar);
        
        \draw (ar) -- (1);
        
        \draw (0) -- (b1);
        \draw (b1) -- (b21);
        \draw (b1) -- (b22);
        \draw (b1) -- (b23);
        
        \draw (b21) -- (b3);
        \draw (b22) -- (b3);
        \draw (b23) -- (b3);
        
        \draw (b3) -- (b4);
        \draw (b4) -- (bdots1);
        \draw[dotted] (bdots1) -- (bdots2);
        \draw (bdots2) -- (br);
        \draw (br) -- (1);
        
        \end{tikzpicture}
        \caption{Counterexample to log-concavity for $\rho(P)\geq 7$}
        \label{fig:non-log-concave}
    \end{figure}
    
\begin{example}
    Let $n>6$ and $m$ be a positive integer. Consider the poset depicted in Figure \ref{fig:non-log-concave}.
    Let us bound the order of magnitude for $h_1,h_2,h_3$ in terms of $m$.
    First, $h_1 \geq m^3$, as $h_1$ is equal to the number of elements of rank greater than $2$.  
    Then, we look at $h_2$. Since we are working with FY monomials, we do not consider products of variables corresponding to incomparable elements.
    The number of rank two FY monomials that do not contain elements of the form $a_{2k,j}$ or $b_{2,j}$ is constant in $m$. Then, each monomial can contain at most one element of the form $b_{2i,j}$ or at most two elements of the form $a_{2k,j}$. In total, we then get that $h_2 \in O(m^4)$. Similarly, $h_3 \geq m^6$, since $\rho(P) \geq 7$ and we can allow three elements of the form $a_{2k,j}$.
    Let us now fix $C$ such that $h_2 \leq C \cdot m^4$. For $m > C^2$, we get that
    \[h_2^2 \leq  Cm^4 \cdot C m^4 < m^9 = m^3 \cdot m^6 \leq h_1h_3,\]
    which violates log-concavity.
    We note that these posets are ranked series-parallel lattices, but none of them is Cohen--Macaulay.
\end{example}

\bibliographystyle{amsalpha}
\bibliography{bibliography}

\end{document}